\documentclass[journal,twoside,web]{ieeecolor}
\usepackage{generic}
\usepackage{textcomp}

\usepackage{booktabs} 
\usepackage{multirow}   
\usepackage{array, makecell} %

\usepackage{amssymb}  
\usepackage{graphicx}
\usepackage{epstopdf}
\usepackage{mathtools}
\usepackage{dsfont}
\usepackage{tikz}
\usepackage{siunitx}
\usepackage[bottom]{footmisc} 
\usepackage{dsfont}
\usepackage{url}
\usepackage{booktabs} 

\usepackage{hyperref}

\usepackage{balance}    

\usepackage{amsthm}

\usepackage{balance}    
\usepackage{xcolor}
\usepackage{tcolorbox}
\usepackage{algpseudocode}
\usepackage{dsfont}
\usepackage{algorithm}

\def\sq{\mathbin{{\strut\rule{1.25ex}{1.25ex}}}}

\newcommand{\beq}{\begin{equation}}
\newcommand{\eeq}{\end{equation}}

\newcommand{\linkCode}{\texttt{\url{https://github.com/urosolia/mixed-observable-LQR}}}

\newcommand{\xk}{\mathbf{x_{k}}}         
\newcommand{\okp}{\mathbf{o_{k+1}}}         
\newcommand{\ok}{\mathbf{o_{k}}}         
\newcommand{\oN}{\mathbf{o_{N}}}         
\newcommand{\okm}{\mathbf{o_{k-1}}}         
\newcommand{\uk}{\mathbf{u}}         
\DeclarePairedDelimiter\floor{\lfloor}{\rfloor}

\newcommand{\ojp}{\mathbf{\bar o_{j(k+1)}}}         
\newcommand{\oj}{\mathbf{\bar o_{j(k)}}}         
\newcommand{\ojN}{\mathbf{\bar o_{j(N)}}}         
\newcommand{\ojm}{\mathbf{\bar o_{j(k-1)}}}         
\newcommand{\ak}{\mathbf{a}}         

\def\mathcolor#1#{\@mathcolor{#1}}
\def\@mathcolor#1#2#3{%
  \protect\leavevmode
  \begingroup
    \color#1{#2}#3%
  \endgroup
}

\newcounter{algorithmctr}[section]
\renewcommand{\thealgorithmctr}{\thesection.\arabic{algorithmctr}}
{\refstepcounter{algorithmctr}\begin{list}{}{%
\setlength{\rightmargin}{0\linewidth}%
\setlength{\leftmargin}{.05\linewidth}
\setlength{\itemsep}{1pt}
\setlength{\parskip}{0pt}
\setlength{\parsep}{0pt}}%
\rmfamily\small
\item[]{\setlength{\parskip}{0ex}\hrulefill\par%
\nopagebreak{\bfseries\textsf{Algorithm \thealgorithmctr~}}}}%
{{\setlength{\parskip}{-1ex}\nopagebreak\par\hrulefill} \end{list}}

\usepackage[noadjust]{cite} 
\usepackage{stmaryrd,graphicx}

\newtheoremstyle{boldStyle}
  {\topsep}
  {\topsep}
  {\itshape}
  {0pt}
  {\bfseries}
  {.}
  { }
  {\thmname{#1}\thmnumber{ #2}\thmnote{ (#3)}}

\newtheoremstyle{italicStyle}
  {\topsep}
  {\topsep}
  {}
  {0pt}
  {\bfseries}
  {.}
  { }
  {\thmname{#1}\thmnumber{ #2}\thmnote{ (#3)}}

\theoremstyle{boldStyle}

\newtheorem{theorem}{Theorem}

\newtheorem{lemma}{Lemma}
\newtheorem{proposition}{Proposition}

\newtheorem{corollary}{Corollary}
\newcommand{\Zsp}{\mathbb{Z}_{0+}}
\newcommand{\Rsp}{\mathbb{R}_{+}}
\newcommand{\Rp}{\mathbb{R}_{0+}}
\newcommand{\Prob}{\mathbb{P}}

\theoremstyle{italicStyle}
\newtheorem{assumption}{Assumption}

\renewenvironment{proof}{{\textbf{Proof:}}}{\hfill$\sq$}

\makeatletter
\newcommand{\fixed@sra}{$\vrule height 2\fontdimen22\textfont2 width 0pt\shortrightarrow$}
\newcommand{\shortarrow}[1]{%
  \mathrel{\text{\rotatebox[origin=c]{\numexpr#1*45}{\fixed@sra}}}
}
\makeatother

\title{\LARGE \bf 
The Mixed-Observable Constrained Linear Quadratic 
Regulator Problem: the Exact Solution and Practical Algorithms}

\author{ Ugo Rosolia, Yuxiao Chen, Shreyansh Daftry, Masahiro Ono, Yisong Yue, and Aaron D. Ames  
\thanks{Ugo Rosolia is with Amazon, 22 Rue Edward Steichen, 2540 Luxembourg, Yuxiao Chen, Yisong Yue, and Aaron D. Ames are at the California Institute of Technology, Pasadena, USA. E-mails: {\tt\scriptsize{\{urosolia\}@amazon.lu}}, {\tt\scriptsize{\{chenyx, yue, ames\}@caltech.edu}}.
Shreyansh Daftry and Masahiro Ono are with the Jet Propulsion Laboratory, California Institute of Technology, Pasadena, USA. E-mails: {\tt\scriptsize{\{daftry, masahiro.ono\}@jpl.nasa.gov}}. This research was carried out 
at the Jet Propulsion Laboratory and the California Institute of Technology 
under a contract with the National Aeronautics and Space Administration.
The authors like to acknowledge also the support of the NSF award \#1932091
.}
}%

\begin{document}

\maketitle
\begin{abstract}
This paper studies the problem of steering a linear system subject to state and input constraints towards a goal location that may be inferred only through noisy partial observations. We assume mixed-observable settings, where the system's state is fully observable and the environment's state defining the goal location is only partially observed. In these settings, the planning problem is an infinite-dimensional optimization problem where the objective is to minimize the expected cost. We show how to reformulate the control problem as a finite-dimensional deterministic problem by optimizing over a trajectory tree. Leveraging this result, we demonstrate that when the environment is static, the observation model piecewise, and cost function convex, the original control problem can be reformulated as a Mixed-Integer Convex Program (MICP) that can be solved to global optimality using a branch-and-bound algorithm. The effectiveness of the proposed approach is demonstrated on navigation tasks, where the goal location should be inferred through noisy measurements.
\end{abstract}

\begin{IEEEkeywords}
Optimal control, observability, measurement uncertainty.
\end{IEEEkeywords}

\section{Introduction}

Model Predictive Control (MPC) is a mature control technology that in part owns its popularity to developments in optimization solvers~\cite{morari1999model, allgower2004nonlinear, mayne2000survey, borrelli2017predictive, wang2009fast, lopez2013fast, jerez2014embedded, kouzoupis2018recent, bemporad2018numerically, gros2020linear}. In MPC, at each time step an optimal planned trajectory is computed solving a finite-dimensional optimization problem, where the cost function and constraints encode the control objectives and safety requirements, respectively. Then, the first optimal control action is applied to the system and the process is repeated at the next time step based on the new measurement. 
This control methodology is ubiquitous in industry, with applications ranging from autonomous driving~\cite{falcone2007predictive, hrovat2012development, lima2018experimental} to large scale power systems~\cite{bengea2012model, serale2018model, maddalena2020data}.

For deterministic discrete-time systems, an optimal trajectory represented by a sequence of states and control actions can be computed leveraging a predictive model of the system. On the other hand, when uncertainties are acting on the system and/or only partial state observations are available, it is not possible to plan an optimal trajectory for the closed-loop system, as its future evolution is uncertain. In these cases, the controller should plan the evolution of the system taking into account that in the future new measurements will be available. More formally, the controller should plan the evolution of the system using a policy that is a function mapping the system's state to a control action. Unfortunately planning over policies is computationally intractable, even for the constrained linear quadratic regulator problem when additive disturbances affect the system's dynamics~\cite{scokaert1998min}. 

Several strategies have been presented in the literature to ease the computational burden of planning over policies~\cite{ goulart2006optimization,  wang2019system, chisci2001systems, mayne2005robust,  yu2013tube, fleming2014robust, liniger2017racing, ben2004adjustable}. When the system dynamics are affected by disturbances and the system's state can be perfectly measured, the planning problem can be simplified by computing affine disturbance feedback policies that map disturbances to control actions~\cite{ben2004adjustable, goulart2006optimization, wang2019system}. 
Another class of feedback policies is considered in tube MPC strategies~\cite{chisci2001systems, mayne2005robust, yu2013tube, fleming2014robust}, where the control actions are computed based on a predefined feedback term and a feed-forward component that is computed online by solving an optimization problem. 
Similar strategies may be used in partially observable settings~\cite{mayne2006robust, alvarado2007output, cannon2012stochastic}. 

The above mentioned strategies are designed for uni-modal disturbances and measurement noise. However, in several practical engineering applications uncertainties are multi-modal, and it is required to design controllers that take the structure of the uncertainty into account to reduce conservatism. For instance, in autonomous driving a controller should plan a trajectory taking into account that surrounding vehicles and pedestrians may exhibit different behaviors that can be categorized into modes, e.g., merging or lane keeping for a car, and crossing or not crossing for a pedestrian~\cite{svensson2018safe, batkovic2020robust, yuxiao2022branch, nair2021stochastic}.
Planning over a trajectory tree, where each branch is associated with different uncertainty modes, is a standard strategy that has been leveraged in the literature to synthesize a controller that can handle multi-model disturbances~\cite{sopasakis2019risk, nair2021stochastic, alsterda2021contingency, alsterda2019contingency, batkovic2020robust, yuxiao2022branch}, when perfect state feedback is available. It is worth mentioning that also adaptive dynamic programming strategies can be used to design controllers for uncertain systems when perfect state feedback is available~\cite{liu2020adaptive, wang2009adaptive, murray2002adaptive}.

In this work, we introduce the mixed-observable constrained linear quadratic regular problem, where perfect state feedback is not available for a subset of the state space. Compared to the standard LQR problem, in our formulation we consider state and input constraints and, most importantly, we assume that only noisy environment measurements about the goal location are available. Thus, the controller has to compute actions also to collect informative measurements. This problem arises in navigation tasks, where a robot has to find an object that could be in a finite number of candidate locations, and the exact one has to be inferred through noisy measurements.
We assume that the system's state is fully observable and we model the partially observable environment state, which represents the goal location, using a hidden Markov model (HMM)~\cite{krishnamurthy2016partially}. The HMM is constructed based on the system and the environment states and it allows us to characterize the observation model by describing the sensors' accuracy. We consider discrete time  systems and environments with continuous and discrete state spaces, respectively.
For this reason, our approach generalizes the strategy from~\cite{ong2009pomdps}, where the authors introduced the mixed-observable control problem for discrete time systems with discrete state spaces.

Our contribution is twofold. First, we show how to reformulate the optimal control problem as a deterministic finite-dimensional optimization problem over a trajectory tree. The computational cost of solving this finite-dimensional optimal control problem increases exponentially with the horizon length, thus we introduce an approximation that can be used to compute a feasible solution to the original problem.
Then, leveraging these results, we demonstrate that through a nonlinear change of coordinates the original optimal control problem can be approximated by solving a Mixed-Integer Convex Program (MICP), when the environment is static and the observation model is piecewise. As a corollary, we show that when the observation model is constant the value function associated with the optimal control problem is convex. Finally, we test the proposed strategy on two navigation examples. 


\textit{Notation: } For a vector $b \in \mathbb{R}^n$ and an integer $s \in \{1,\ldots, n\}$, we denote $b[s]$ as the $s$-th component of the vector $b$, $b^\top$ indicates its transpose, $M = \text{diag}(b) \in \mathbb{R}^{n \times n}$ is a diagonal matrix with diagonal elements $M[s, s] = b[s]$, and $v= 1/b$ is defined as a vector $v \in \mathbb{R}^n$ with entries $v[s] = 1/b[s]$ for all $s \in \{1,\ldots, n\}$. For a function $T : \mathbb{R}^n \rightarrow \mathbb{R}$, $T(b)$ denotes the value of the function $T$ at $b$. Throughout the paper, we will use capital letters to indicate functions and lower letters to indicate vectors. 
The set of positive integers is denoted as $\Zsp = \{1, 2, \ldots\}$, and the set of (strictly) positive reals as ($\Rsp = (0, \infty)$) $\Rp = [0, \infty)$. Furthermore, given a set $\mathcal{Z}$ and an integer $k$, we denote the $k$-th Cartesian product as $\mathcal{Z}^k=\mathcal{Z}\times\ldots\times\mathcal{Z}$ and $|\mathcal{Z}|$ as the cardinality of $\mathcal{Z}$. Finally, given a real number $a\in\mathbb{R}$ we define the floor function $\floor{a}$, which outputs the largest integer $i = \floor{a}$ such that $i \leq a$. 

\section{Problem Formulation}\label{sec:problemFormulation}


\subsection{System and Environment Models}\label{sec:envSysModel}
We consider the following linear time-invariant system:
\begin{equation}\label{eq:sys}
    x_{k+1} = A x_k + B u_k,
\end{equation}
where the state $x_k \in \mathbb{R}^n$, the input $u_k \in \mathbb{R}^d$, and $k$ indexes over discrete time steps. Furthermore, the above system is subject to the following state and input constraints: 
\begin{equation}\label{eq:cnstr}
     u_k \in \mathcal{U} \subseteq \mathbb{R}^d\text{ and } x_k \in \mathcal{X}\subseteq \mathbb{R}^n, \forall k\geq 0.
\end{equation}

Our goal is to control system~\eqref{eq:sys} in environments represented by partially observable discrete states. In particular, the environment evolution is modeled using a hidden Markov model (HMM) given by the tuple $\mathcal{H} = \left( \mathcal{E}, \mathcal{O},T, Z \right)$, where: 
\begin{itemize}

	\item $\mathcal{E}=\{1,\ldots,|\mathcal{E}|\}$ is a set of partially observable environment states;
    
	\item $\mathcal{O}=\{1, \ldots,|\mathcal{O}|\}$ is the set of observations.
			
    \item The function $T: \mathcal{E} \times \mathcal{E} \times \mathbb{R}^n \rightarrow [0,1]$ describes the probability of transitioning to a state $e'$ given the current environment state $e$ and system's state $x$, i.e., $T(e', e, x) :=\mathbb{P}(e'| e, x)$. 

	\item The function $Z: \mathcal{E} \times \mathcal{O}\times \mathbb{R}^n \rightarrow [0,1]$ describes the probability of observing $o$, given the environment state $e$ and the system's state $x$, i.e., $Z(e,o,x) :=  \Prob(o| e, x)$.

\end{itemize}

As the environment state $e_k$ is partially observable, we introduce the following belief vector:
\begin{equation*}
    b_k \in \mathcal{B} = \{b \in \mathbb{R}^{|\mathcal{E}|}_{0+} : \sum_{e=1}^{|\mathcal{E}|} b[e] = 1\}.
\end{equation*}
The belief $b_k$ is a sufficient statistics and each entry $b_k[e]$ represents the posterior probability that the state of the environment $e_k$ equals $e\in\mathcal{E}$, given the observation vector $\ok = [o_1, \ldots,o_k]$, the system's trajectory $\xk = [x_1, \ldots, x_k]$, the state $x(0)$, and the belief vector $b(0)$ at time $t=0$, i.e., $b_k[e] = \Prob(e|\ok, \xk, x(0), b(0))$.

Consider an example where a Mars rover has to find a science sample that may be in one of several locations, which are identified using coarse and low resolution surface images~\cite{Daftry2022MLnav, rosoliaMOMDP2021}. As the exact location is unknown, the rover is required to collect measurements to identify the science sample's location.
In this setting, the environment could be represented by an HMM where the set of environment states $\mathcal{E}$ collects the possible science sample locations, e.g., $\mathcal{E}= \{\tt{loc}_1, \ldots, \tt{loc}_n \}$ and $e = \tt{loc}_i$ if the science sample is in the $i$-th location. In the next section, we further formalize this navigation task as a regulation problem.

\subsection{Control Objectives}
Given the environment's belief $b(t)$ and system's state $x(t)$, our goal is solve the following \textit{finite time optimal control problem (FTOCP)}:
\begin{equation}\label{eq:probGoal}
\begin{aligned}
J(x(t), b(t)) ~&\\= \min_{\boldsymbol{\pi}} \quad & \mathbb{E}_{\mathbf{o_{N-1}}}\bigg[\sum_{k=0}^{N-1} h(x_k, u_k, e_k) + h_N(x_N, e_N)\bigg| b(t)\bigg] \\
    \text{subject to} \quad &    x_{k+1} = A x_k + B u_k,\\
    & u_k = \pi_k(\ok, \xk, x(t), b(t)), \\
    & x_0 = x(t), \\
    & u_k \in \mathcal{U}, x_k \in \mathcal{X},\forall k\in\{0,\ldots,N-1\},
\end{aligned}
\end{equation}
where the stage cost $h : \mathbb{R}^n \times \mathbb{R}^d \times \mathcal{E} \rightarrow \mathbb{R}$ and the terminal cost $h_N : \mathbb{R}^n \times \mathcal{E} \rightarrow \mathbb{R}$. Note that the objective is a function of the partially observable environment states $e_k \in \mathcal{E}$, and the expectation is over the environment observations $\mathbf{o_{N-1}} = [o_1,\ldots, o_{N-1}]$, which are stochastic as discussed in Section~\ref{sec:envSysModel}. In the above FTOCP, the optimization is carried out over the sequence of control policies $\boldsymbol{\pi} = [\pi_0, \ldots, \pi_{N-1}]$, and at each time $k$ the policy $\pi_k : \mathcal{O}^{k} \times \mathcal{X}^{k+1} \times  \mathcal{B}\rightarrow \mathbb{R}^d$ maps the environment observations up to time $k$, the system's trajectory, and the initial belief $b(t)$ to the control action $u_k$. Notice that we focus on the solution to the above finite-time control task and we do not analyze the stability properties of the closed-loop system.

Computing the optimal solution to the FTOCP~\eqref{eq:probGoal} is challenging as $i)$ the environment's state is partially observable, $ii)$ our goal is to minimize the expected cost, and $iii)$ the optimization is infinite dimensional as it is carried out over the space of feedback policies, which are functions mapping states and belief vectors to inputs. 
In what follows, we show that the FTOCP~\eqref{eq:probGoal} can be reformulated as a finite-dimensional non-linear program (NLP). Leveraging the discrete nature of the set of observations $\mathcal{O}$, we will show that optimizing over feedback policies is equivalent to optimizing over a tree of control actions. Furthermore, we show that when the environment is static, the cost functions $h(\cdot, \cdot, e)$ and $h_N(\cdot, e)$ are convex and quadratic, and the observation function $Z(e,o, \cdot):\mathbb{R}^n \rightarrow [0,1]$ is piecewise for all $e \in \mathcal{E}$ and $o \in \mathcal{O}$, then the FTOCP~\eqref{eq:probGoal} can be recast as an MICP. Finally, we show that when the observation model is constant the FTOCP~\eqref{eq:probGoal} can be written as a convex parametric optimization problem.

\section{The Exact Solution}\label{sec:exact}

\subsection{Cost Reformulation}
As discussed in Section~\ref{sec:envSysModel}, the belief $b_k$ is a sufficient statistics for an HMM~\cite{krishnamurthy2016partially}. Therefore, at each time $k$ the belief can be computed using the observation $o_k$, the system's state $x_k$, and the belief at the previous time step $b_{k-1}$, i.e., 
\begin{equation}\label{eq:beliefUpdateComponent}
\begin{aligned}
    b_k[e] 
    &= \frac{Z(e, o_k, x_k)}{\Prob(o_k |x_k,  b_{k-1})}\sum_{i\in \mathcal{E}} T(e,i, x_k) b_{k-1}[i].
\end{aligned}
\end{equation}
For further details about the belief update equation please refer to~\cite{ong2009pomdps, rosoliaMOMDP2021}. 
The above equation can be written in compact form:
\begin{equation*}
\begin{aligned}
    b_k = \frac{A_e(o_k, x_k) b_{k-1} }{\Prob(o_k |x_k,  b_{k-1})},
\end{aligned}
\end{equation*}
where $\Prob(o_k |x_k,  b_{k-1})$ is a normalization constant and the matrix $A_e(o_k, x_k) \in \mathbb{R}^{|\mathcal{E}| \times |\mathcal{E}|}$, which is a function of the observations $o_k$ and the system's state $x_k$ at time $k$, is defined as follows:
\begin{equation}\label{eq:Amatrix}
    A_e(o_k, x_k) = \Theta(o_k, x_k) \Omega(x_k),
\end{equation}
where 
\begin{equation}\label{eq:omega_def}
    \Omega(x_k) = \begin{bmatrix} T(1,1, x_k) & \ldots & T(1,|\mathcal{E}|, x_k)\\ T(2,1, x_k) & \ldots & T(2,|\mathcal{E}|, x_k)\\
    \vdots & & \vdots \\    
    T(|\mathcal{E}|,1,x_k) & \ldots & T(|\mathcal{E}|,|\mathcal{E}|, x_k)
    \end{bmatrix}
\end{equation}
and
\begin{equation*}
    \Theta(o_k, x_k) = \text{diag}\Big(\begin{bmatrix} Z(1, o_k, x_k) & \ldots & Z(|\mathcal{E}|, o_k, x_k)
    \end{bmatrix}\Big).
\end{equation*}

Leveraging the above definitions, we show that the expected cost from problem~\eqref{eq:probGoal} can be rewritten as a summation over the set of possible observations $\mathcal{O}$.

\begin{proposition}
Consider the optimal control problem~\eqref{eq:probGoal}. The expected cost can be equivalently written as 
\begin{equation}\label{eq:costRef}
   \begin{aligned}
       \mathbb{E}_{\mathbf{o_{N-1}}}\Bigg[ \sum_{k=0}^{N-1} &h(x_k, u_k, e_k) + h_N(x_N, e_N)\Bigg| b_0 \Bigg] \\
       &=  \sum_{k=0}^{N-1} \sum_{\ok \in \mathcal{O}^k} \sum_{e \in \mathcal{E}} v^{\ok}_k[e] h(x_k, u_k, e)\\
       &\quad\quad\quad+\sum_{\oN \in \mathcal{O}^N} \sum_{e \in \mathcal{E}} v^{\oN}_N[e] h_N(x_N, e),
   \end{aligned}
\end{equation}
where the unnormalized belief $v^{\ok}_k = A_e(o_k, x_k) v^{\okm}_{k-1}$ and the matrix $A_e(o_k, x_k) \in \mathbb{R}^{|\mathcal{E}| \times |\mathcal{E}|}$ is defined in~\eqref{eq:Amatrix}.
\end{proposition}

\begin{proof}
First, we notice that, as the system dynamics are deterministic, the expected stage cost at time step $k$ can be written as:
\begin{equation}\label{eq:costDerivation}
    \begin{aligned}
     &\mathbb{E}_{\mathbf{o_{N-1}}}[  h(x_k, u_k, e_k) | \xk, x_0, b_0 ]  \\
    &=\sum_{\ok \in \mathcal{O}^k} \mathbb{E}_{\mathbf{o_{N-1}}} [ h(x_k, u_k, e_k) | \xk, x_0, b_0, \ok ] \Prob(\ok|\xk, x_0, b_0) \\
    &=\sum_{\ok \in \mathcal{O}^k} \sum_{e \in \mathcal{E}} b_k[e] h(x_k, u_k, e) \Prob(\ok|\xk, x_0, b_0) \\
            & = \sum_{\ok \in \mathcal{O}^k} \sum_{e \in \mathcal{E}} v^{\ok}_k[e] h(x_k, u_k, e).
    \end{aligned}
\end{equation}
In the above derivation we leveraged the independence of the observations collected at each time step, i.e., $\Prob(\ok|\xk, x_0, b_0)=\Prob(o_1|x_1, x_0, b_0)\times\ldots\times\Prob(o_k|\xk, x_0, b_0)$, and we defined
\begin{equation*}
     v^{\ok}_k[e] = {Z(e, o_k, x_k)}\sum_{i\in\mathcal{E}} T(e,i, x_k) v^{\okm}_{k-1}[i],
\end{equation*}
which can be written in compact form as 
    $v^{\ok}_k = A_e(o_k, x_k) v^{\okm}_{k-1}$.
Finally, we notice that the derivation in~\eqref{eq:costDerivation} holds also for the terminal cost function $h_N$. Therefore,
we have that the desired result follows from~\eqref{eq:costDerivation} and the linearity of the expectation in equation~\eqref{eq:costRef}.
\end{proof}

\begin{figure}[b!]
    \centering
	\includegraphics[width= 1.0\columnwidth]{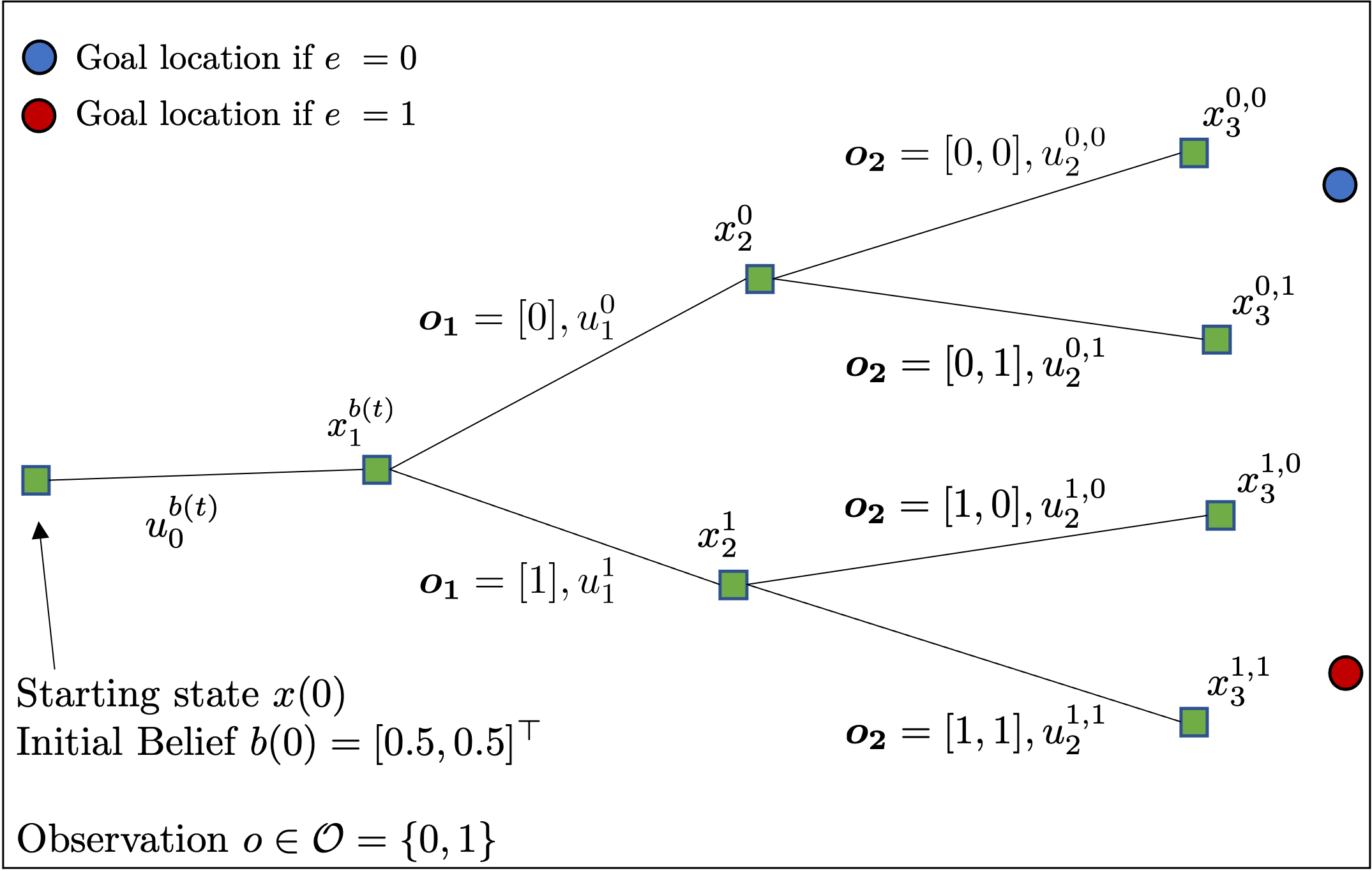}
    \vspace{-0.5cm}
    \caption{Tree of trajectories for $N=3$, where at each time $k$ there are $|\mathcal{O}|=2$ possible observations. 
    Each predicted control action $u_k^{\ok}$ is associated with an observation vector $\ok \in \mathcal{O}^k$. Thus, the above tree encodes a policy given by the actions that the controller would apply in the future depending on the observations collected up to time $k$.
    }\label{fig:tree}
\end{figure}

\subsection{Deterministic Reformulation}
In the previous section, we showed how to leverage the beliefs associated with all possible observations to express the expectation as a summation. In this section, we show that the optimization carried out over feedback policies can be reformulated as an optimization over a tree of trajectories, as the one shown in Figure~\ref{fig:tree}.

The control policy $\pi_k : \mathcal{O}^{k} \times \mathcal{X}^{k+1} \times  \mathcal{B}$ from~\eqref{eq:probGoal} maps the vector of observations $\ok =[o_1, \ldots, o_k] \in \mathcal{O}^{k}$, the system's trajectory, and the initial belief $b_0=b(0)$ to the control action $u_k$, i.e., $ u_k = \pi(\ok, \xk, x_0, b_0)$.
Notice that the system dynamics from problem~\eqref{eq:probGoal} are deterministic and therefore, given an initial condition $x(t)$ and an initial belief $b(t)$, the control action at time $k$ is a function only of the observation vector $\ok$. Thus, we define the control action $u_k^{\ok} \in \mathbb{R}^d$ associated with the observation vector $\ok \in \mathcal{O}^k$, and we reformulate problem~\eqref{eq:probGoal} as an optimization problem over the set of control actions $\{ u_k^{\ok} \in \mathbb{R}^d : k \in \{0, \ldots, N-1\},  \ok \in \mathcal{O}^k\}$. This strategy allows us to optimize over policies as at time $k$ the controller plans $|O|^k$ distinct actions associated with each uncertain sequence of observations  $\ok = [o_1, \ldots, o_k] \in \mathcal{O}^k$. Basically,  the controller optimizes over a tree of control actions,  as shown in Figure~\ref{fig:tree}.
More formally, given the environment's belief $b(t)$ and the system's state $x(t)$, we rewrite problem~\eqref{eq:probGoal} as:


\begin{equation}\label{eq:probDet}
\begin{aligned}
J(x(t), b(t)) = \min_{\uk} \quad & \sum_{k=0}^{N-1} \sum_{\ok \in \mathcal{O}^k} \sum_{e \in \mathcal{E}} v^{\ok}_k[e] h(x_k^{\ok}, u^{\ok}_k, e) \\
&\quad\quad + \sum_{\oN \in \mathcal{O}^N} \sum_{e \in \mathcal{E}} v^{\oN}_N[e] h_N(x_N^{\oN}, e)\\
    \text{subject to} \quad~ &    x_{k+1}^{\ok} = A x_k^{\okm} + B u_k^{\ok},\\
    & x^{\mathbf{o_{-1}}}_0 = x(t), v^{\mathbf{o_0}}_0 = b(t), \\
    & v^{\okp}_{k+1} = A_e(o_{k+1}, x_{k+1}^{\ok}) v^{\ok}_{k}, \\
    & u_k^{\ok} \in \mathcal{U}, x_{k+1}^{\ok} \in \mathcal{X},\\
    & \forall {\ok} \in \mathcal{O}^k, \forall k\in\{0,\ldots,N-1\},
\end{aligned}
\end{equation}
where the vector of observations $\ok = [{{o_1}}, \ldots, {{o_k}}]$ for all $k \in \{1, \ldots, N-1\}$, and at time $k=0$ we defined $\mathbf{o}_0 = \mathbf{o}_{-1} = b(t)$, and $\mathcal{O}^0 = b(t)$.
In the above problem, the matrix of decision variables is defined as:
\begin{equation}\label{eq:varProbDet}
    \uk = [u_0^{\mathbf{o_0}}, \ldots, u_{N-1}^{\mathbf{o_{N-1}}}] \in \mathbb{R}^{d \times \sum_{k=0}^{N-1}|\mathcal{O}|^k}.
\end{equation}
Note that, for each time step $k \in \{0, \ldots, N-1\}$, the above matrix collects the $|\mathcal{O}|^k$ control actions associated with all observation vectors from the set $\mathcal{O}^k$.

\begin{lemma}\label{lemma:sol}
Assume that $\mathcal{X}$ and $\mathcal{U}$ are compact. Let $\mathcal{C} \subset \mathcal{X}$ be a control invariant set for system~\eqref{eq:sys} subject to constraints~\eqref{eq:cnstr}, i.e., $\forall x \in \mathcal{X}$ there exists $u \in \mathcal{U}$ such that $Ax + Bu \in \mathcal{C}$. If $x(t) \in \mathcal{C}$, then problem~\eqref{eq:probDet} admits an optimal solution.
\end{lemma}
\begin{proof}
By definition, we have that for $x(t) \in \mathcal{C}$ there exists a sequence of $N$ control actions that keep the system inside $\mathcal{C}$. Hence, problem~\eqref{eq:probDet} is feasible. Compactness of state and input constraint sets yields the desired result.
\end{proof}


\subsection{Practical Approach}\label{sec:pracApproach}
The FTOCP~\eqref{eq:probDet} is a finite-dimensional NLP that can be solved with off-the-self solvers. However, the computational cost of solving~\eqref{eq:probDet} is non-polynomial in the horizon length, as the number of decision variables from~\eqref{eq:varProbDet} grows exponentially with the horizon length $N$. Indeed, at each time step $k$ the predicted trajectory branches as a function of the discrete observation $o_k \in \mathcal{O}$, as shown in Figure~\ref{fig:tree}. In this section, we introduce an approximation to the FTOCP~\eqref{eq:probDet}, where the predicted trajectory branches every $N_b$ time steps. This strategy allows us to limit the number of optimization variables and, for a prediction horizon of $N$ steps, the computational burden is proportional to the ratio $N/N_b$. 

Given the current state $x(t)$, the environment's belief $b(t)$, the constant $N_b \in \Zsp$, and the prediction horizon $N = P N_b$ with $P\in\Zsp$, we solve the following FTOCP:
\begin{equation}\label{eq:probPrac}
\begin{aligned}
\hat J(x(t), b(t))& \\= \min_{ \ak} ~ & \sum_{k=0}^{N-1} \sum_{\oj \in \mathcal{O}^{j(k)}} \sum_{e \in \mathcal{E}} \bar v^{\oj}_k[e] h(s_k^{\oj}, a^{\oj}_k, e) \\
    & \quad \quad \quad + \sum_{\ojN \in \mathcal{O}^{j(N)}} \sum_{e \in \mathcal{E}} \bar v^{\ojN}_k[e] h_N(s_N^{\ojN}, e)\\
    \text{subject to} ~~ &    s_{k+1}^{\oj} = A s_k^{\ojm} + B a_k^{\oj},\\
    & s^{\mathbf{\bar o_{-1}}}_0 = x(t), \bar v^{\mathbf{\bar o_0}}_0 = b(t), \\
    & \bar v^{\ojp}_{k+1} = C_e(\bar o_{j(k+1)}, s_{k+1}^{\oj}, k) \bar v^{\oj}_{k}, \\
    & a_k^{\oj} \in \mathcal{U}, s_{k+1}^{\oj} \in \mathcal{X},\\
    & j(k) = \floor{k/N_b}, \\
    & \forall {\oj} \in \mathcal{O}^{j(k)}, \forall k\in\{0, \ldots, N-1\},
\end{aligned}
\end{equation}
where for $P = N/N_b \in \Zsp$ and $j(k) = \floor{k/N_b}$ the matrix of decision variables
\begin{equation}\label{eq:varProbPrac}
    \ak \! = \! [a_0^{\mathbf{\bar o}_0}, \ldots, a_P^{\mathbf{\bar o}_0}, \ldots, a_{k}^{\mathbf{\bar o_{j(k)}}}\!, \ldots, a_{N-1}^{\mathbf{\bar o_{j(N-1)}}}] \! \in \! \mathbb{R}^{d \times \sum_{k=0}^{P-1} N_b|\mathcal{O}|^k}\!\!,
\end{equation}
the vector of observations $\mathbf{\bar o_{j(N-1)}} = \mathbf{\bar o_{P-1}}=[\bar o_{1}, \ldots, \bar o_{P-1}]$, and the matrix $C_e(\bar o_k, s_k^{\ojm}, k)$ is defined as:
\begin{equation}\label{eq:Cmatrix}
\begin{aligned}
    C_e(&\bar o_{j(k)},  s_k^{\ojm}, k) \\
    &= \begin{cases}
    A_e(\bar o_{j(k)}, s_k^{\ojm}) & \mbox{If } \floor{k/N_b} = t/N_b \text{ and } k>0\\
    \Omega(s_k^{\ojm}) & \mbox{Otherwise}
    \end{cases},
\end{aligned}
\end{equation}
where $\Omega(\cdot)$ is defined as in~\eqref{eq:omega_def}.

Compare the FTOCP~\eqref{eq:probDet} with the FTOCP~\eqref{eq:probPrac}. In the FTOCP~\eqref{eq:probDet}  we optimize over the tree of trajectories shown in Figure~\ref{fig:tree}, and therefore the complexity of the problem grows exponentially with the horizon length $N$. 
On the other hand, in the FTOCP~\eqref{eq:probPrac} we optimize over a tree of trajectories that branches every $N_b$ time steps, and the matrix of optimization variables~\eqref{eq:varProbPrac} grows exponentially with the ratio $P = N/N_b$. 
Therefore, in the FTOCP~\eqref{eq:probPrac} the user-defined constant $N_b$ may be used to limit the computational complexity when planning over a horizon $N$. As a trade-off, the optimal value function $\hat J$ associated with the  FTOCP~\eqref{eq:probPrac} only approximates the value function $J$ associated with the FTOCP~\eqref{eq:probDet}. 

\section{Static Environments, Piecewise Observation Model, and Quadratic Cost: the Exact Solution}\label{sec:MICP}
In this section, we consider problems with static environments, piecewise observation model, and convex quadratic cost function. Under these assumptions, we show that problem~\eqref{eq:probDet} can be reformulated as an MICP.

In what follows, we first introduce the problem setup. Then, we show how to reformulate problem~\eqref{eq:probDet} as an MICP.

\begin{assumption}[Static Environment]\label{ass:static} The environment is static, which in turns implies that the transition function $T$ is defined as follows: $T(e,e) = 1, T(e',e) = 0, \forall e\in\mathcal{E}$ and $\forall e'\in\mathcal{E}$ such that $ e\neq e'$.
\end{assumption}


\begin{assumption}[Piecewise Observation Model]\label{ass:pwa} The observation model is a piecewise function of the system state $x$. In particular, given $R$ disjointed polytopic regions $\{\mathcal{X}_i\}_{i=1}^R$ such that $\cup_{i=1}^R \mathcal{X}_i = \mathcal{X}$, we have that: $Z(e, o, x) = M_i(e, o) \text{ if } x \in \mathcal{X}_i$,
for a set of functions $M_i:\mathcal{E} \times \mathcal{O} \rightarrow [0, 1]$.
\end{assumption}

\begin{assumption}[Convex Quadratic Cost Function]\label{ass:cost}
For a fixed environment state $e \in \mathcal{E}$, the stage cost $h( \cdot, \cdot, e) : \mathbb{R}^n \times \mathbb{R}^d \rightarrow \mathbb{R}$ and the terminal cost $h_N(\cdot, e) : \mathbb{R}^n \rightarrow \mathbb{R}$ are convex and quadratic, i.e., $h( x, u, e) = ||x - x_g^{(e)}||_Q + ||u-u_g^{(e)}||_R, h_N(x, e) = ||x - x_g^{(e)}||_{Q_N}$
where the weighted square norm $||x||_Q = x^\top Q x$ for the positive semi-definite matrix $Q$, and the vectors $x_g^{(e)} \in \mathbb{R}^n$ and $u_g^{(e)} \in \mathbb{R}^d$ are user-defined.
\end{assumption}

\begin{assumption}[Strictly Positive Belief]\label{ass:belifStrictlyPositive}
All entries of the belief vector $b(0)$ are strictly positive, i.e., $b(0) \in \mathcal{B}_+ = \{b \in \mathbb{R}^{|\mathcal{E}|}_{0+} : \sum_{i=1}^{|\mathcal{E}|}b[e]=1\}$. Furthermore, we cannot observe the true environment state $e\in\mathcal{E}$ from any state $x \in \mathcal{X}$, i.e., $\mathbb{P}(o=e|e , x) = Z(e, o, x) : \mathcal{E} \times \mathcal{O} \times \mathbb{R}^n \rightarrow (0,1)$.
\end{assumption}

Given the system's state $x(t)$ and the inverse belief vector $z(t) = 1/b(t) \in \mathbb{R}^{|\mathcal{E}|}$, we define the following FTOCP:
\begin{equation}\label{eq:probDetPWA}
\begin{aligned}
V(x(t), z(t)) = \min_{\uk, \boldsymbol{\delta}} \quad & \sum_{k=0}^{N-1} \sum_{\ok \in \mathcal{O}^k} \sum_{e \in \mathcal{E}} \frac{ h(x_k^{\ok}, u^{\ok}_k, e)}{z^{\ok}_k[e]} \\
    & \quad \quad \quad \quad + \sum_{\oN \in \mathcal{O}^N} \sum_{e \in \mathcal{E}} \frac{ h_N(x_N^{\oN}, e)}{z^{\oN}_N[e]}\\
    \text{subject to} \quad~ &    x_{k+1}^{\ok} = A x_k^{\okm} + B u_k^{\ok},\\
    & x^{\mathbf{o_{-1}}}_0 = x(t), z^{\mathbf{o_{0}}}_0 = z(t), \\
    & u_k^{\ok} \in \mathcal{U}, x_k^{\ok} \in \mathcal{X},\\
    & z^{\okp}_{k+1} = \textstyle \sum_{i=1}^R D_i(o_{k+1}) z^{\ok}_{k} \delta^{\ok}_{k,i} ,\\
    & \delta^{\ok}_{k,i} = \mathds{1}_{\mathcal{X}_i}(x_k^{\ok}), \forall i \in \{1,\ldots, R\},\\
    &\forall k\in\{0,\ldots,N-1\},
\end{aligned}
\end{equation}
where the indicator function $\mathds{1}_{\mathcal{X}_i}(x_k^{\ok}) = 1$ if $x_k^{\ok} \in \mathcal{X}_i$ and zero otherwise,
and the optimization variables
\begin{equation}
\begin{aligned}
    \uk &=  [u_0^{\mathbf{o_0}}, \ldots, u_{N-1}^{\mathbf{o_{N-1}}}] \in \mathbb{R}^{d \times \textstyle \sum_{k=0}^{N-1}|\mathcal{O}|^k}, \\
    \boldsymbol{\delta} &= [\delta_{0,1}^{\mathbf{o_0}}, \ldots, \delta_{N-1, R}^{\mathbf{o_{N-1}}}] \in \{0, 1\}^{ R \textstyle \sum_{k=0}^{N-1}|\mathcal{O}|^k}.
\end{aligned}
\end{equation}
Notice that at each time $k$ for the vector of observations $\ok$, we have that the integer variable $\delta_{k,i}^{\mathbf{o_k}}$ equals one if and only if the state $x_{k}^{\mathbf{o_k}} \in \mathcal{X}_i$.
In the above problem, for all $i \in \{1,\ldots, R\}$ the entries of diagonal matrices $D_i(o) \in \mathbb{R}^{|\mathcal{E}|\times |\mathcal{E}|}$ are defined as follows:
\begin{equation}\label{eq:defInverse}
    D_i(o)[e,e] = 1/M_i(o,e), \forall e \in \mathcal{E}, \forall o \in \mathcal{O}.
\end{equation}

The following theorem shows that, under Assumptions~\ref{ass:static}--\ref{ass:belifStrictlyPositive}, problem~\eqref{eq:probDetPWA} is equivalent to problem~\eqref{eq:probGoal}. Furthermore, problem~\eqref{eq:probDetPWA} can be recast as an MICP.

\begin{theorem}\label{th:equivalence}
Consider problem~\eqref{eq:probGoal} and problem~\eqref{eq:probDetPWA}. Let Assumptions~\ref{ass:static}--\ref{ass:belifStrictlyPositive} hold. Then, for $z(t) = 1/b(t)$ we have that 
\begin{equation*}
    J(x(t), b(t)) = V(x(t), z(t)),
\end{equation*}
for all $x(t) \in \mathcal{X}$. Furthermore, for all $z(t) \in \mathbb{R}_+^{|\mathcal{E}|}$ and $x(t) \in \mathbb{R}^n$ problem~\eqref{eq:probDetPWA} can be recast as a Mixed-Integer Convex Program (MICP).
\end{theorem}
\begin{proof} First we show that $z_k^{\ok} = 1/v^{\ok}_k$ for all $k \in \{0, \ldots, N-1\}$. 
From Assumptions~\ref{ass:static}--\ref{ass:pwa}, we have that for $x_k \in \mathcal{X}_i$ the unnormalized belief update is
\begin{equation}
\begin{aligned}
     v^{\ok}_k[e] & = {Z(e, o_k,x_k)}\sum_{i\in\mathcal{E}} T(e,i) v^{\okm}_{k-1}[i]\\
     & ={Z(e, o_k,x_k)} v^{\okm}_{k-1}[e]\\
     & = M_i(e, o_k) v^{\okm}_{k-1}[e].
\end{aligned}
\end{equation}
From the above equation and definition~\eqref{eq:defInverse}, we have that $z_k^{\ok}[e] = 1/v^{\ok}_k[e], \forall e \in \mathcal{E}$,
which in turns implies that the optimal cost from problem~\eqref{eq:probDetPWA} equals the one from problem~\eqref{eq:probDet} and therefore
\begin{equation*}
    J(x(t), b(t)) = V(x(t), z(t)),
\end{equation*}
for all $x(t) \in \mathcal{X}$.

Notice that the objective function in problem~\eqref{eq:probDetPWA} is convex, as it is given by a convex quadratic function over a strictly positive linear function~\cite{boyd2004convex}. Furthermore, given the initial condition $z(t)$ we can compute an upper-bound $z_{k}^{\textrm{max}}[e]$ for each $e$-th entry of the unnormalized belief $z_k^\ok$, i.e.,
\begin{equation}\label{eq:uppBound}
    z_{k}^{\textrm{max}}[e] = \Big(\max_{ \substack{ o \in \mathcal{O}, i \in \{1, \ldots, R\}}} D_i(o) \Big)^{k-1} z_0^{\mathbf{o_0}}[e] \geq z_k^\ok[e].
\end{equation}
Finally, we have that the piecewise model from Assumption~\ref{ass:pwa} is a mixed logical dynamical (MLD) systems~\cite{bemporad1999control}. Thus following the procedure presented in~\cite{bemporad1999control}, problem~\eqref{eq:probDetPWA} can be recast as an MICP using the upper-bound from~\eqref{eq:uppBound}.
\end{proof}

\begin{corollary}\label{cor:cvx}
Consider problem~\eqref{eq:probDetPWA} and let Assumptions~\ref{ass:static}--\ref{ass:belifStrictlyPositive} hold. If the observation model is not a function of the system's state, i.e., for some $G:\mathcal{E} \times \mathcal{O} \rightarrow [0, 1]$ we have that 
\begin{equation*}
    Z(e, o, x) = G(e,o), \forall x \in \mathcal{X}.
\end{equation*}
Then, the value function $V(x(t), z(t))$ from problem~\eqref{eq:probDetPWA} is convex in its arguments.
\end{corollary}

\begin{proof}
As the observation model does not dependent on the system's state, we have that the belief update in problem~\eqref{eq:probDetPWA} can be re-written as follows: $z^{\okp}_{k+1} = F(o_{k+1}) z^{\ok}_{k}$, where $F(o)[e,e] = 1/G(o,e)$ for all $e \in \mathcal{E}$ and $o \in \mathcal{O}$.
Therefore, problem~\eqref{eq:probDetPWA}
is a convex parametric program and $V(x(t), z(t))$ is a convex function~\cite{bemporad2006algorithm}.
\end{proof}



\section{Examples}\label{sec:examples}
We tested the proposed strategy on two navigation problems, where a linear system has to reach a goal location that may be inferred only through partial observations. The goal location represents an object to be retrieved and whose location is only partially known. We consider the following discrete time unstable point mass model:
\begin{equation}\label{eq:doubleIntegrator}
    x_{k+1} = \begin{bmatrix} 1 & 0 & 1 & 0\\
                              0 & 1 & 0 & 1\\
                              0 & 0 & 1.1 & 0\\
                              0 & 0 & 0 & 1.1
    \end{bmatrix}x_k + \begin{bmatrix}
    0 & 0\\
    0 & 0\\
    1 & 0\\
    0 & 1
    \end{bmatrix}u_k,
\end{equation}
where the state vector $x_k = [X_k, Y_k, v^x_k, v^y_k]$ collects the position of the system $(X_k, Y_k)$ and the velocity $(v^x_k, v^y_k)$ along the $X$--$Y$ plane. In the above system, the input $u_k = [a_k^x, a^y_k]$ represents the accelerations along the $X$ and $Y$ coordinates.

\subsection{Mixed Observable Regulation Problem}
In this example, the constraint sets are defined as follows:
\begin{equation*}
\begin{aligned}
    &\mathcal{U} = \{u \in \mathbb{R}^2 : ||u||_\infty \leq 10\}, \\
    &\mathcal{X} = \{[X, Y, v^x, v^y]^\top  \in \mathbb{R}^4 : -5 \leq X \leq 15, ||Y||_\infty \leq 10\},
\end{aligned}
\end{equation*}
and the cost matrices from Assumption~\ref{ass:cost} are
\begin{equation*}
    Q = 10^{-5}I_n, R = 10^{-3}I_d, \text{ and } Q_N = 10^{2}I_n,
\end{equation*}
where $I_n \in \mathbb{R}^{n \times n}$ represents the identity matrix.

The set of partially observable states $\mathcal{E} = \{0, 1\}$ and the associated goal locations $x_g^{(0)} = [14, 8, 0, 0]^
\top$ and $x_g^{(1)} = [14,-8, 0, 0]^\top$, as shown in Figures~\ref{fig:res_Scenario1}--\ref{fig:res_Scenario2}.
The environment state, and consequently the goal location, is inferred through partial observations. Given the true environment state $e \in \mathcal{E}$ and the system's state $x \in \mathbb{R}^n$, the probability of measuring the observation $o = e$ is given by the following piecewise observation model:
\begin{equation}\label{eq:res_obsModel}
    Z(o=e, e, x)=\mathbb{P}(o=e|e,x) = \begin{cases}p_1 & \mbox{If } x \in \mathcal{X}_1 \\
    p_2 = 0.85 & \mbox{If } x \in \mathcal{X}_2 \end{cases},
\end{equation}
where 
\begin{equation*}
\begin{aligned}
    \mathcal{X}_1 &= \{[X, Y, v^x, v^y] \in \mathbb{R}^4 : -1 \leq X \leq ~15, ||Y||_\infty \leq 10 \} \\
    \mathcal{X}_2 &= \{[X, Y, v^x, v^y] \in \mathbb{R}^4 : -5 \leq X < -1, ||Y||_\infty \leq 10 \}.
\end{aligned}
\end{equation*}

We implemented the finite-dimensional MICP~\eqref{eq:probDetPWA} using CVXPY~\cite{diamond2016cvxpy} and Gurobi~\cite{gurobi}. In order to limit the computational burden, we leveraged the strategy discussed in Section~\ref{sec:pracApproach} for $N=60$, and $N_b \in \{12, 15, 20, 30\}$. All computations are run on a 2015 MacBook Pro and the code can be found at~\linkCode.

\begin{figure}[t!]
    \centering
	\includegraphics[width= 1.00\columnwidth,trim=30 10 30 10,clip]{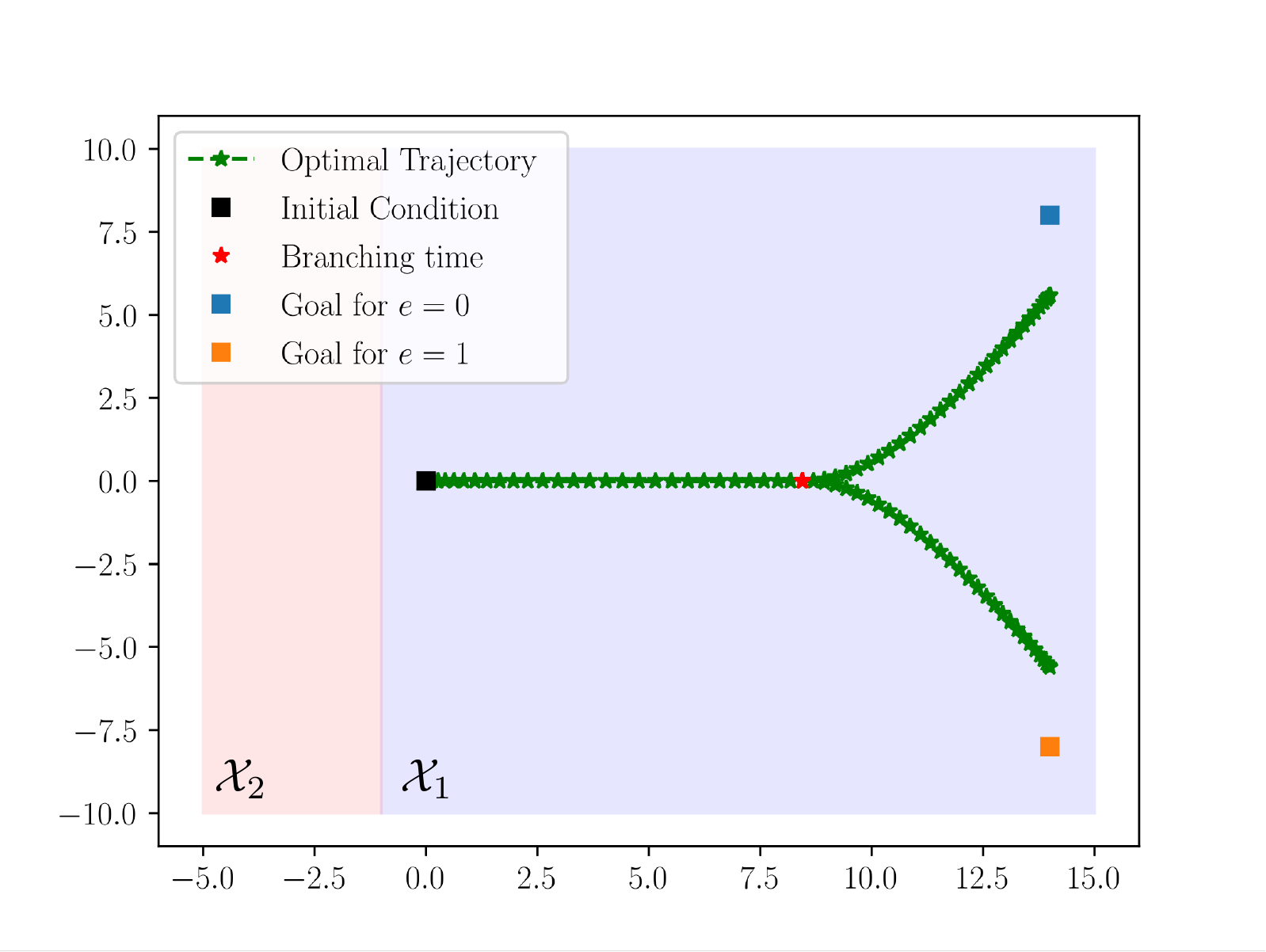}
    \vspace{-0.7cm}  
    \caption{Optimal trajectory computed solving the MICP~\eqref{eq:probDetPWA} for $N=60$ and assuming that an observation is collected every $N_b = 30$ time steps, as discussed in Section~\ref{sec:pracApproach}. In this scenario $p_1 = p_2 = 0.85$, therefore the optimizer computes a trajectory that first steers the system towards the goals and then commits to one of the two goal locations depending on observation measured at time $t=N_b$.}
    \label{fig:res_Scenario1}
        \vspace{-0.4cm}  
\end{figure}

\begin{figure}[t!]
    \centering
	\includegraphics[width= 1.0\columnwidth,trim=30 10 30 10,clip]{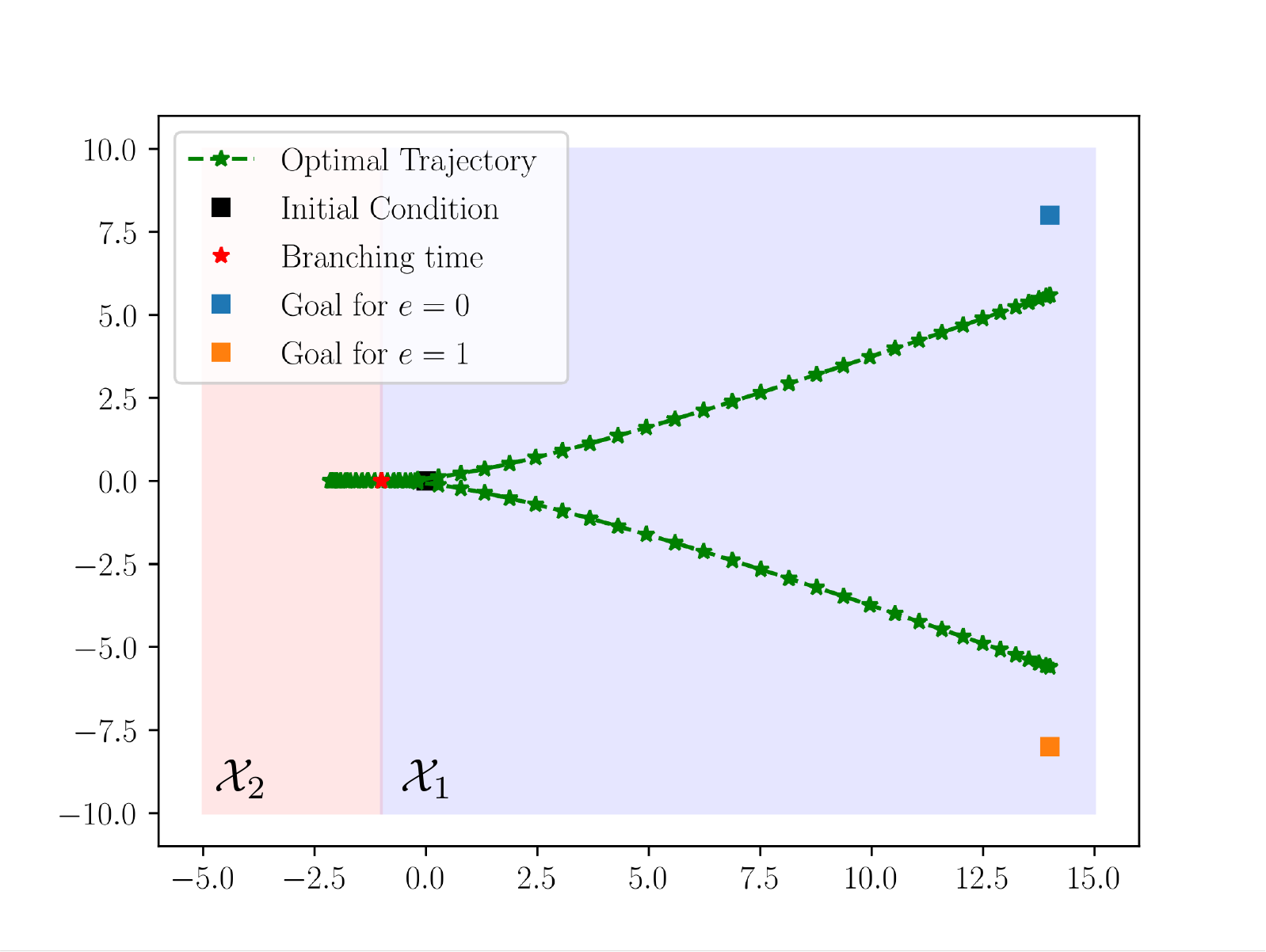}
    \vspace{-0.7cm}  
    \caption{Optimal trajectory computed solving the MICP~\eqref{eq:probDetPWA} for $N=60$ and assuming that an observation is collected every $N_b = 30$ time steps, as discussed in Section~\ref{sec:pracApproach}. In this scenario $p_1 = 0.7$, therefore the controller steers the system backwards to reach region $\mathcal{X}_2$ to collect a measurement that is correct with probability $p_2 = 0.85$, before committing to a goal location.}
    \label{fig:res_Scenario2}
        \vspace{-0.4cm}  
\end{figure}

\begin{table}[h!]
\centering
\caption{Optimal cost $V(x(0), b(0))$ and solver time for different values of $N_b$
and consequently of $P=N/N_b$.
}\label{tab:comparison}
\begin{tabular}{c|cccc}
                &   $N_b = 12$   & $N_b = 15$ &   $N_b = 20$ & $N_b = 30$ \\\midrule
$V(x(0), b(0))$ &   $1237.43$    & $1583.31$  &   $2196.75$  & $3265.31$   \\
Solver Time [s] &    $134.1$     &  $12.1$    &   $2.8$      & $1.7$      \\
$P=N/N_b$       &    $5$         &  $4$       &   $3$        & $2$        \\\midrule
\end{tabular}
\end{table}

We tested the proposed strategy for two different scenarios. In the first scenario, we set the probability $p_1$ of the observation model~\eqref{eq:res_obsModel} equal to $0.85$, and in the second one we set $p_1 = 0.7$. In both cases, we considered as initial condition $x(0) = [0, 0, 0, 0]^\top$, an initial belief $b(0) = [0.5, 0.5]^\top$, a prediction horizon $N=60$, and we assumed that an observation is collected every $N_b=30$ time steps. In the first scenario shown in Figure~\ref{fig:res_Scenario1}, the probability $p_1 = p_2 = 0.85$ and the observations collected in regions $\mathcal{X}_1$ and $\mathcal{X}_2$ are equally informative. Thus, the optimizer steers the system forward, and after collecting an observation at time $t=N_b$ commits to a goal location. On the other hand, when $p_1 = 0.7$ the observation collected in region $\mathcal{X}_1$ is not as informative as the one collected in region $\mathcal{X}_2$. Therefore, the optimizer plans a trajectory that moves backward and visits region $\mathcal{X}_2$ to collect an observation that is correct with probability $p_2 = 0.85$, before committing to steer the system towards a goal location, as shown in Figure~\ref{fig:res_Scenario2}.

Table~\ref{tab:comparison} shows the optimal cost and the computational time to solve the MICP for different values of $N_b$ and for $N = 60$. As discussed in Section~\ref{sec:pracApproach}, as $P = N/N_b$ gets larger the optimization tree has more branches and, consequently, the problem complexity increases. In particular, the number of optimization variables $v = \sum_{k=0}^{P-1} N_b|\mathcal{O}|^k$ grows exponentially as a function of $P$.  

\subsection{Partially Observable Navigation Problem}
We test the proposed strategy on the navigation task shown in Figures~\ref{fig:res_ex2_Scenario1}--\ref{fig:res_ex2_Scenario2}. 
In this example, there are two obstacles (black regions) and the objective is to reach a goal location that can only be inferred through partial observations. The observation model is piecewise and it is defined as follows:
\begin{equation}\label{eq:res_obsModel_1}
    Z(o = e, e, x)=\mathbb{P}(o = e| e, x) = \begin{cases}
    p_1 = 0.5  & \mbox{If } x \in \mathcal{X}_1 \\
    p_2 = 0.7 & \mbox{If } x \in \mathcal{X}_2 \\
    p_3 = 0.85 & \mbox{If } x \in \mathcal{X}_3 \\
    p_4 = 0.85 & \mbox{If } x \in \mathcal{X}_4 
    \end{cases},
\end{equation}
where regions $\mathcal{X}_1$, $\mathcal{X}_2$, $\mathcal{X}_3$, and $\mathcal{X}_4$ are depicted in Figures~\ref{fig:res_ex2_Scenario1}--\ref{fig:res_ex2_Scenario2}. Less formally, the observation function in~\eqref{eq:res_obsModel_1} models the accuracy of the sensors that are more accurate when the system is close to the candidate goal location and there is no occlusion caused by the obstacles. Indeed, observations collected in region $\mathcal{X}_1$ are not informative; on the other hand, in region $\mathcal{X}_2$ the probability that an observation is correct is $p_2 = 0.7$, and the most informative observations are collected in regions $\mathcal{X}_3$ and $\mathcal{X}_4$.
Finally, we consider the unstable point mass model~\eqref{eq:doubleIntegrator} and the cost function is defined by the following matrices $Q = 10^{-4}I_n, R = 10^{-2}I_d, \text{ and } Q_N = 10 I_n$
where $I_n \in \mathbb{R}^{n \times n}$ represents the identity matrix.

\begin{figure}[t!]
    \centering
	\includegraphics[width= 1.0\columnwidth,trim=30 10 30 10,clip]{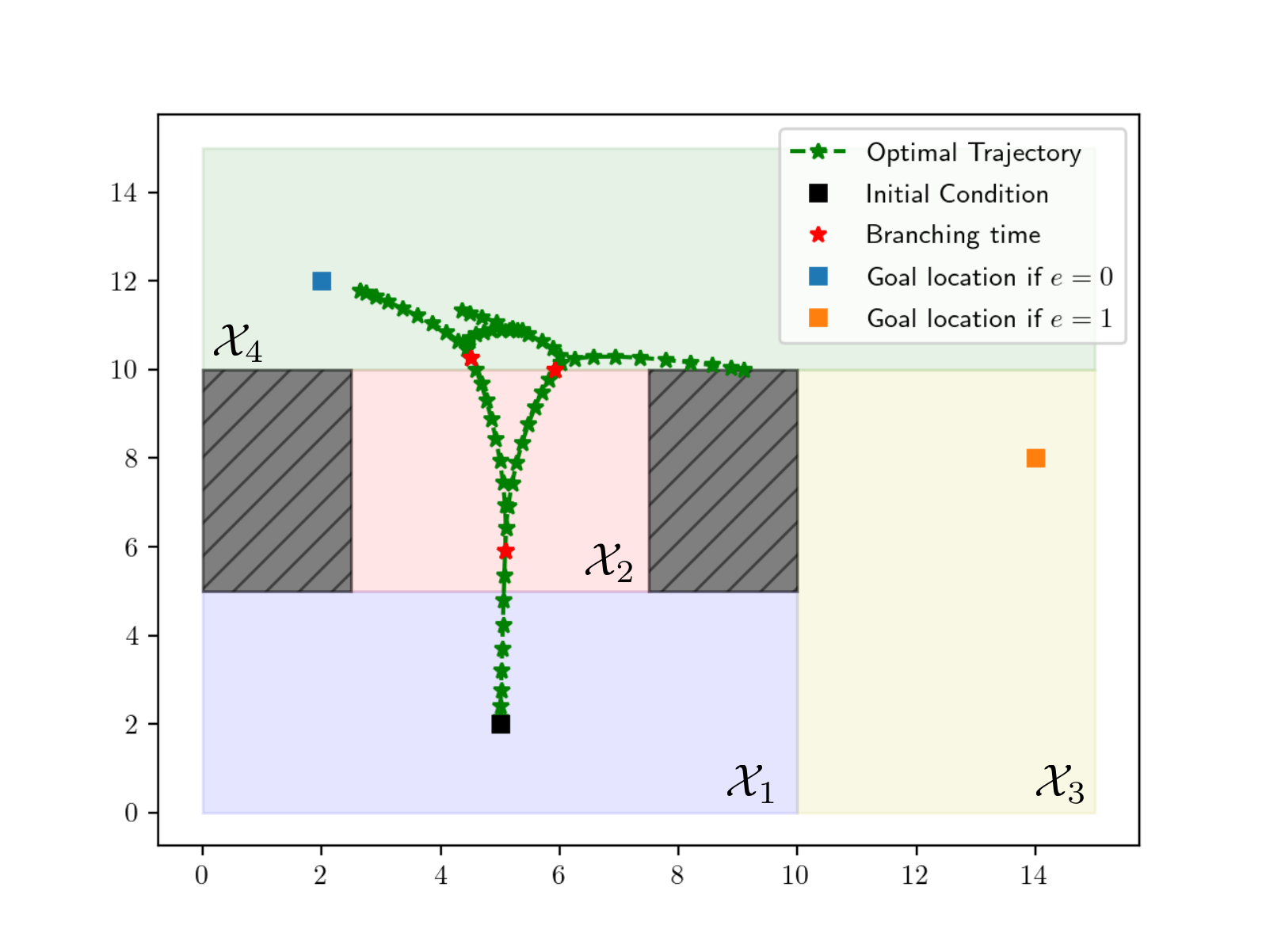}
    \vspace{-0.7cm}  
    \caption{Optimal trajectory computed solving the MICQ for $N=30$ and $N_b = 10$. The objective is to steer the system to the goal location that is a function of the partially observable state $e \in \{0, 1\}$, while avoiding the two obstacles (black rectangles). In this scenario, the initial belief $b(0) = [0.8, 0.2]^\top$ and the observation model is piecewise over the regions $\mathcal{X}_1$, $\mathcal{X}_2$, $\mathcal{X}_2$, and $\mathcal{X}_4$.}
    \label{fig:res_ex2_Scenario1}
            \vspace{-0.4cm}  
\end{figure}

\begin{figure}[t!]
    \centering
	\includegraphics[width= 1.0\columnwidth,trim=30 10 30 10,clip]{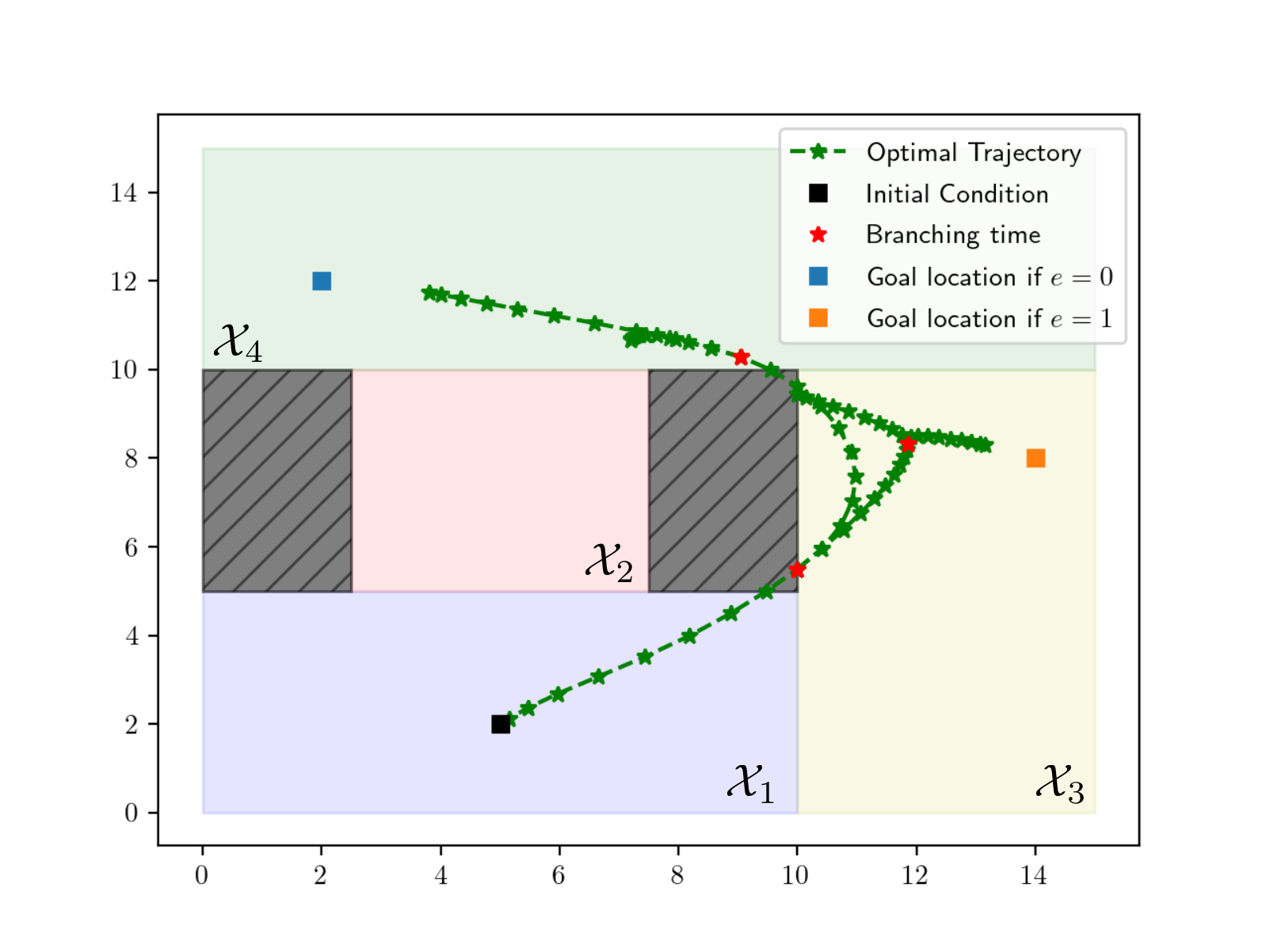}
    \vspace{-0.76cm}  
    \caption{Optimal trajectory computed solving the MICQ for $N=30$ and $N_b = 10$. The objective is to steer the system to the goal location that is a function of the partially observable state $e \in \{0, 1\}$, while avoiding the two obstacles (black rectangles). In this scenario, the initial belief $b(0) = [0.5, 0.5]^\top$ and the observation model is piecewise over the regions $\mathcal{X}_1$, $\mathcal{X}_2$, $\mathcal{X}_2$, and $\mathcal{X}_4$.}
    \label{fig:res_ex2_Scenario2}
        \vspace{-0.4cm}  
\end{figure}

We implemented the MICP using CVXPY~\cite{diamond2016cvxpy}. Notice that the feasible regions is non-convex as there are two obstacles in the environment. For this reason, at time $k$ we introduced integer variables to constraint the state of the system $x_k$  to lie in either $\mathcal{X}_1$, $\mathcal{X}_2$, $\mathcal{X}_3$, or $\mathcal{X}_4$. For implementation details please refer to the source code available at~\linkCode.

We tested the proposed strategy for two initial belief vectors. In both scenarios, we set a prediction horizon $N=30$ and the parameter $N_b=10$. Therefore, the optimal trajectory computed solving the MICP branches at time $t=10$ and time $t=20$. Figure~\ref{fig:res_ex2_Scenario1} shows the optimal trajectory tree when the initial belief $b(0) = [0.8, 0.2]^\top$. 
Notice that, as we have a strong belief that the environment state $e = 0$, the controller plans a trajectory tree that goes through region $\mathcal{X}_2$ to reach the goal location associated with the state $e=0$. 
On the other hand, when the initial belief $b(0) = [0.5, 0.5]^\top$, the optimizer plans a trajectory that collects observations only in regions $\mathcal{X}_3$ and $\mathcal{X}_4$, as shown in Figure~\ref{fig:res_ex2_Scenario2}. This result is expected as when we do not have any prior knowledge about the goal location--in this example $b(0)=[0.5, 0.5]^\top$--an optimal strategy should maximize the number of informative observations that are collected in regions $\mathcal{X}_3$ and~$\mathcal{X}_4$.


\section{Conclusions}
In this work, we introduced the mixed-observable constrained linear quadratic regulator problem, where the goal of the controller is to steer the system to a goal location that may be inferred only through partial observations. We showed that when the system's state space is continuous and the environment's state is discrete, the control problem can be reformulated as a finite-dimensional optimization problem over a trajectory tree. Leveraging this result, we showed that under mild assumptions the control problem can be recast as an MICP through a nonlinear change of coordinates. 

\bibliographystyle{IEEEtran} 
\bibliography{IEEEabrv,mybibfile}

\begin{thebibliography}{10}
\providecommand{\url}[1]{#1}
\csname url@samestyle\endcsname
\providecommand{\newblock}{\relax}
\providecommand{\bibinfo}[2]{#2}
\providecommand{\BIBentrySTDinterwordspacing}{\spaceskip=0pt\relax}
\providecommand{\BIBentryALTinterwordstretchfactor}{4}
\providecommand{\BIBentryALTinterwordspacing}{\spaceskip=\fontdimen2\font plus
\BIBentryALTinterwordstretchfactor\fontdimen3\font minus
  \fontdimen4\font\relax}
\providecommand{\BIBforeignlanguage}[2]{{%
\expandafter\ifx\csname l@#1\endcsname\relax
\typeout{** WARNING: IEEEtran.bst: No hyphenation pattern has been}%
\typeout{** loaded for the language `#1'. Using the pattern for}%
\typeout{** the default language instead.}%
\else
\language=\csname l@#1\endcsname
\fi
#2}}
\providecommand{\BIBdecl}{\relax}
\BIBdecl

\bibitem{morari1999model}
M.~Morari and J.~H. Lee, ``Model predictive control: past, present and
  future,'' \emph{Computers \& Chemical Engineering}, vol.~23, no. 4-5, pp.
  667--682, 1999.

\bibitem{allgower2004nonlinear}
F.~Allgower, R.~Findeisen, Z.~K. Nagy \emph{et~al.}, ``Nonlinear model
  predictive control: From theory to application,'' \emph{Journal-Chinese
  Institute Of Chemical Engineers}, vol.~35, no.~3, pp. 299--316, 2004.

\bibitem{mayne2000survey}
D.~Q. Mayne, J.~B. Rawlings, C.~V. Rao, and P.~O. Scokaert, ``Survey
  constrained model predictive control: Stability and optimality,''
  \emph{Automatica (Journal of IFAC)}, vol.~36, no.~6, pp. 789--814, 2000.

\bibitem{borrelli2017predictive}
F.~Borrelli, A.~Bemporad, and M.~Morari, \emph{Predictive control for linear
  and hybrid systems}.\hskip 1em plus 0.5em minus 0.4em\relax Cambridge
  University Press, 2017.

\bibitem{wang2009fast}
Y.~Wang and S.~Boyd, ``Fast model predictive control using online
  optimization,'' \emph{IEEE Transactions on control systems technology},
  vol.~18, no.~2, pp. 267--278, 2009.

\bibitem{lopez2013fast}
R.~Lopez-Negrete, F.~J. D’Amato, L.~T. Biegler, and A.~Kumar, ``Fast
  nonlinear model predictive control: Formulation and industrial process
  applications,'' \emph{Computers \& Chemical Engineering}, vol.~51, pp.
  55--64, 2013.

\bibitem{jerez2014embedded}
J.~L. Jerez, P.~J. Goulart, S.~Richter, G.~A. Constantinides, E.~C. Kerrigan,
  and M.~Morari, ``Embedded online optimization for model predictive control at
  megahertz rates,'' \emph{IEEE Transactions on Automatic Control}, vol.~59,
  no.~12, pp. 3238--3251, 2014.

\bibitem{kouzoupis2018recent}
D.~Kouzoupis, G.~Frison, A.~Zanelli, and M.~Diehl, ``Recent advances in
  quadratic programming algorithms for nonlinear model predictive control,''
  \emph{Vietnam Journal of Mathematics}, vol.~46, no.~4, pp. 863--882, 2018.

\bibitem{bemporad2018numerically}
A.~Bemporad and V.~V. Naik, ``A numerically robust mixed-integer quadratic
  programming solver for embedded hybrid model predictive control,''
  \emph{IFAC-PapersOnLine}, vol.~51, no.~20, pp. 412--417, 2018.

\bibitem{gros2020linear}
S.~Gros, M.~Zanon, R.~Quirynen, A.~Bemporad, and M.~Diehl, ``From linear to
  nonlinear {MPC}: bridging the gap via the real-time iteration,''
  \emph{International Journal of Control}, vol.~93, no.~1, pp. 62--80, 2020.

\bibitem{falcone2007predictive}
P.~Falcone, F.~Borrelli, J.~Asgari, H.~E. Tseng, and D.~Hrovat, ``Predictive
  active steering control for autonomous vehicle systems,'' \emph{IEEE
  Transactions on control systems technology}, vol.~15, no.~3, pp. 566--580,
  2007.

\bibitem{hrovat2012development}
D.~Hrovat, S.~Di~Cairano, H.~E. Tseng, and I.~V. Kolmanovsky, ``The development
  of model predictive control in automotive industry: A survey,'' in \emph{2012
  IEEE International Conference on Control Applications}.\hskip 1em plus 0.5em
  minus 0.4em\relax IEEE, 2012, pp. 295--302.

\bibitem{lima2018experimental}
P.~F. Lima, G.~C. Pereira, J.~M{\aa}rtensson, and B.~Wahlberg, ``Experimental
  validation of model predictive control stability for autonomous driving,''
  \emph{Control Engineering Practice}, vol.~81, pp. 244--255, 2018.

\bibitem{bengea2012model}
S.~Bengea, A.~Kelman, F.~Borrelli, R.~Taylor, and S.~Narayanan, ``Model
  predictive control for mid-size commercial building {HVAC}: Implementation,
  results and energy savings,'' in \emph{Second international conference on
  building energy and environment}, 2012, pp. 979--986.

\bibitem{serale2018model}
G.~Serale, M.~Fiorentini, A.~Capozzoli, D.~Bernardini, and A.~Bemporad, ``Model
  predictive control ({MPC}) for enhancing building and {HVAC} system energy
  efficiency: Problem formulation, applications and opportunities,''
  \emph{Energies}, vol.~11, no.~3, p. 631, 2018.

\bibitem{maddalena2020data}
E.~T. Maddalena, Y.~Lian, and C.~N. Jones, ``Data-driven methods for building
  control—a review and promising future directions,'' \emph{Control
  Engineering Practice}, vol.~95, p. 104211, 2020.

\bibitem{scokaert1998min}
P.~O. Scokaert and D.~Q. Mayne, ``Min-max feedback model predictive control for
  constrained linear systems,'' \emph{IEEE Transactions on Automatic control},
  vol.~43, no.~8, pp. 1136--1142, 1998.

\bibitem{goulart2006optimization}
P.~J. Goulart, E.~C. Kerrigan, and J.~M. Maciejowski, ``Optimization over state
  feedback policies for robust control with constraints,'' \emph{Automatica},
  vol.~42, no.~4, pp. 523--533, 2006.

\bibitem{wang2019system}
Y.-S. Wang, N.~Matni, and J.~C. Doyle, ``A system-level approach to controller
  synthesis,'' \emph{IEEE Transactions on Automatic Control}, vol.~64, no.~10,
  pp. 4079--4093, 2019.

\bibitem{chisci2001systems}
L.~Chisci, J.~A. Rossiter, and G.~Zappa, ``Systems with persistent
  disturbances: predictive control with restricted constraints,''
  \emph{Automatica}, vol.~37, no.~7, pp. 1019--1028, 2001.

\bibitem{mayne2005robust}
D.~Q. Mayne, M.~M. Seron, and S.~Rakovi{\'c}, ``Robust model predictive control
  of constrained linear systems with bounded disturbances,'' \emph{Automatica},
  vol.~41, no.~2, pp. 219--224, 2005.

\bibitem{yu2013tube}
S.~Yu, C.~Maier, H.~Chen, and F.~Allg{\"o}wer, ``Tube {MPC} scheme based on
  robust control invariant set with application to {L}ipschitz nonlinear
  systems,'' \emph{Systems \& Control Letters}, vol.~62, no.~2, pp. 194--200,
  2013.

\bibitem{fleming2014robust}
J.~Fleming, B.~Kouvaritakis, and M.~Cannon, ``Robust tube {MPC} for linear
  systems with multiplicative uncertainty,'' \emph{IEEE Transactions on
  Automatic Control}, vol.~60, no.~4, pp. 1087--1092, 2014.

\bibitem{liniger2017racing}
A.~Liniger, X.~Zhang, P.~Aeschbach, A.~Georghiou, and J.~Lygeros, ``Racing
  miniature cars: Enhancing performance using stochastic {MPC} and disturbance
  feedback,'' in \emph{2017 American Control Conference (ACC)}.\hskip 1em plus
  0.5em minus 0.4em\relax IEEE, 2017, pp. 5642--5647.

\bibitem{ben2004adjustable}
A.~Ben-Tal, A.~Goryashko, E.~Guslitzer, and A.~Nemirovski, ``Adjustable robust
  solutions of uncertain linear programs,'' \emph{Mathematical programming},
  vol.~99, no.~2, pp. 351--376, 2004.

\bibitem{mayne2006robust}
D.~Q. Mayne, S.~V. Rakovi{\'c}, R.~Findeisen, and F.~Allg{\"o}wer, ``Robust
  output feedback model predictive control of constrained linear systems,''
  \emph{Automatica}, vol.~42, no.~7, pp. 1217--1222, 2006.

\bibitem{alvarado2007output}
I.~Alvarado, D.~Lim{\'o}n, T.~Alamo, and E.~F. Camacho, ``Output feedback
  robust tube based {MPC} for tracking of piece-wise constant references,'' in
  \emph{2007 46th IEEE Conference on Decision and Control}.\hskip 1em plus
  0.5em minus 0.4em\relax IEEE, 2007, pp. 2175--2180.

\bibitem{cannon2012stochastic}
M.~Cannon, Q.~Cheng, B.~Kouvaritakis, and S.~V. Rakovi{\'c}, ``Stochastic tube
  {MPC} with state estimation,'' \emph{Automatica}, vol.~48, no.~3, pp.
  536--541, 2012.

\bibitem{svensson2018safe}
L.~Svensson, L.~Masson, N.~Mohan, E.~Ward, A.~P. Brenden, L.~Feng, and
  M.~T{\"o}rngren, ``Safe stop trajectory planning for highly automated
  vehicles: An optimal control problem formulation,'' in \emph{2018 IEEE
  Intelligent Vehicles Symposium (IV)}.\hskip 1em plus 0.5em minus 0.4em\relax
  IEEE, 2018, pp. 517--522.

\bibitem{batkovic2020robust}
I.~Batkovic, U.~Rosolia, M.~Zanon, and P.~Falcone, ``A robust scenario {MPC}
  approach for uncertain multi-modal obstacles,'' \emph{IEEE Control Systems
  Letters}, vol.~5, no.~3, pp. 947--952, 2020.

\bibitem{yuxiao2022branch}
Y.~Chen, U.~Rosolia, W.~Ubellacker, N.~Csomay-Shanklin, and A.~D. Ames,
  ``Interactive multi-modal motion planning with branch model predictive
  control,'' \emph{IEEE Robotics and Automation Letters}, vol.~7, no.~2, pp.
  5365--5372, 2022.

\bibitem{nair2021stochastic}
S.~H. Nair, V.~Govindarajan, T.~Lin, C.~Meissen, H.~E. Tseng, and F.~Borrelli,
  ``Stochastic {MPC} with multi-modal predictions for traffic intersections,''
  \emph{arXiv preprint arXiv:2109.09792}, 2021.

\bibitem{sopasakis2019risk}
P.~Sopasakis, D.~Herceg, A.~Bemporad, and P.~Patrinos, ``Risk-averse model
  predictive control,'' \emph{Automatica}, vol. 100, pp. 281--288, 2019.

\bibitem{alsterda2021contingency}
J.~P. Alsterda and J.~C. Gerdes, ``Contingency model predictive control for
  linear time-varying systems,'' \emph{arXiv preprint arXiv:2102.12045}, 2021.

\bibitem{alsterda2019contingency}
J.~P. Alsterda, M.~Brown, and J.~C. Gerdes, ``Contingency model predictive
  control for automated vehicles,'' in \emph{2019 American Control Conference
  (ACC)}.\hskip 1em plus 0.5em minus 0.4em\relax IEEE, 2019, pp. 717--722.

\bibitem{liu2020adaptive}
D.~Liu, S.~Xue, B.~Zhao, B.~Luo, and Q.~Wei, ``Adaptive dynamic programming for
  control: A survey and recent advances,'' \emph{IEEE Transactions on Systems,
  Man, and Cybernetics: Systems}, vol.~51, no.~1, pp. 142--160, 2020.

\bibitem{wang2009adaptive}
F.-Y. Wang, H.~Zhang, and D.~Liu, ``Adaptive dynamic programming: An
  introduction,'' \emph{IEEE computational intelligence magazine}, vol.~4,
  no.~2, pp. 39--47, 2009.

\bibitem{murray2002adaptive}
J.~J. Murray, C.~J. Cox, G.~G. Lendaris, and R.~Saeks, ``Adaptive dynamic
  programming,'' \emph{IEEE transactions on systems, man, and cybernetics, Part
  C (Applications and Reviews)}, vol.~32, no.~2, pp. 140--153, 2002.

\bibitem{krishnamurthy2016partially}
V.~Krishnamurthy, \emph{Partially observed Markov decision processes}.\hskip
  1em plus 0.5em minus 0.4em\relax Cambridge university press, 2016.

\bibitem{ong2009pomdps}
S.~C. Ong, S.~W. Png, D.~Hsu, and W.~S. Lee, ``Pomdps for robotic tasks with
  mixed observability.'' in \emph{Robotics: Science and Systems}, vol.~5, 2009,
  p.~4.

\bibitem{Daftry2022MLnav}
S.~Daftry, N.~Abcouwer, T.~D. Sesto, S.~Venkatraman, J.~Song, L.~Igel, A.~Byon,
  U.~Rosolia, Y.~Yue, and M.~Ono, ``Mlnav: Learning to safely navigate on
  martian terrains,'' \emph{IEEE Robotics and Automation Letters}, vol.~7,
  no.~2, pp. 5461--5468, 2022.

\bibitem{rosoliaMOMDP2021}
U.~Rosolia, M.~Ahmadi, R.~M. Murray, and A.~D. Ames, ``Time-optimal navigation
  in uncertain environments with high-level specifications,'' in \emph{2021
  60th IEEE Conference on Decision and Control (CDC)}, 2021, pp. 4287--4294.

\bibitem{boyd2004convex}
S.~Boyd and L.~Vandenberghe, \emph{Convex optimization}.\hskip 1em plus 0.5em
  minus 0.4em\relax Cambridge university press, 2004.

\bibitem{bemporad1999control}
A.~Bemporad and M.~Morari, ``Control of systems integrating logic, dynamics,
  and constraints,'' \emph{Automatica}, vol.~35, no.~3, pp. 407--427, 1999.

\bibitem{bemporad2006algorithm}
A.~Bemporad and C.~Filippi, ``An algorithm for approximate multiparametric
  convex programming,'' \emph{Computational optimization and applications},
  vol.~35, no.~1, pp. 87--108, 2006.

\bibitem{diamond2016cvxpy}
S.~Diamond and S.~Boyd, ``{CVXPY}: {A} {P}ython-embedded modeling language for
  convex optimization,'' \emph{Journal of Machine Learning Research}, vol.~17,
  no.~83, pp. 1--5, 2016.

\bibitem{gurobi}
\BIBentryALTinterwordspacing
{Gurobi Optimization, LLC}, ``{Gurobi Optimizer Reference Manual},'' 2021.
  [Online]. Available: \url{https://www.gurobi.com}
\BIBentrySTDinterwordspacing

\end{thebibliography}

\end{document}